\documentclass[11pt]{amsart}

\usepackage{latexsym,amsmath,amssymb,amsmath,mathtools,tikz,microtype,comment}
\usepackage[english]{babel}
\usepackage[margin=2.5cm]{geometry}
\usepackage{enumitem,MnSymbol}
\usepackage{floatrow}
\usepackage{multicol}

\usepackage[colorlinks=true, citecolor=blue, urlcolor=magenta, breaklinks=true]{hyperref}

\numberwithin{equation}{section}
\newtheorem{theorem}{Theorem}[section]
\newtheorem{lemma}[theorem]{Lemma}

\newtheorem{corollary}[theorem]{Corollary}

\theoremstyle{definition}
\newtheorem{definition}[theorem]{Definition} 
\newtheorem{remark}[theorem]{Remark}
\newtheorem{example}[theorem]{Example}

\DeclareMathOperator{\init}{in}

\DeclareMathOperator{\pdim}{pdim}
\DeclareMathOperator{\reg}{reg}
\DeclareMathOperator{\depth}{depth}

\newcommand{\K}{{\mathbb{K}}}

\newcommand{\Z}{{\mathbb{Z}}}

\begin{document}

%%%%%%%%%%%%%%%%%%%%%%%%%%%%%%%%%%%%%%%%%%%%%%%%%%%%%%%%%%%%%%%% 
 
\title{Regularity and $h$-polynomials of toric ideals of graphs}
\thanks{Version: \today}

\author[G. Favacchio]{Giuseppe Favacchio}
\author[G. Keiper]{Graham Keiper}
\author[A. Van Tuyl]{Adam Van Tuyl}

\address[G. Favacchio]
{Dipartimento di Matematica e Informatica\\
Universit\`a degli Studi di Catania\\
Viale A. Doria, 6 \\
95100 - Catania, Italy}
\email{favacchio@dmi.unict.it }

\address[G. Keiper, A. Van Tuyl]
{Department of Mathematics and Statistics\\
McMaster University, Hamilton, ON, L8S 4L8}
\email{keipergt@mcmaster.ca, 
vantuyl@math.mcmaster.ca}


 %%%%%%%%%%%%%%%%%%%%%%%%%%%%%%%%%%%%%%%%%%%%%%%%%%%%%%%%%%%%%%%

\begin{abstract}
For all integers $4 \leq r \leq d$, we show that there 
exists a finite simple graph $G= G_{r,d}$ with toric ideal
$I_G \subset R$ such that $R/I_G$ has
(Castelnuovo-Mumford) regularity $r$ and $h$-polynomial
of degree $d$.  
To achieve this goal, we identify a family of graphs
such that the graded Betti numbers of the associated toric ideal
agree with its initial ideal, and furthermore, this initial ideal
has linear quotients.  As a corollary, we can 
recover a result of Hibi, Higashitani, Kimura, and O'Keefe that
compares the depth and dimension of toric ideals of graphs.
\end{abstract}

\thanks{{\bf Acknowledgements.} 
The authors thank Johannes Hofscheier for his feedback.  
Our results were inspired by
calculations using {\em Macaualy2} \cite{M2}. Favacchio thanks
the support of the Universit\`a degli Studi di
Catania ``Piano della Ricerca 2016/2018 
Linea di intervento 2" and the ``National  Group for Algebraic and Geometric Structures, and their Applications" (GNSAGA-INdAM).   Van
Tuyl acknowledges the support of NSERC RGPIN-2019-05412.}

\maketitle

%%%%%%%%%%%%%%%%%%%%%%%%%%%%%%%%%%%%%%%%%%%%%%%%%%%%%%%%%%%%%%%%

\section{Introduction}\label{s. intro and back on HS}

Let $\K$ be an algebraically closed field of characteristic zero, and let $R = \K[x_1,\ldots,x_n]$ be the standard graded polynomial ring over $\K$.   Hibi and Matsuda \cite{HM1}
initiated a comparison of the (Castelnuovo-Mumoford) regularity
of $R/I$ and the degree of the $h$-polynomial appearing 
in the Hilbert series of $R/I$.  They
showed that for any integers $d,r \geq 1$, there is a monomial ideal
$I$ such that the regularity of $R/I$ is $r$, and the degree of the $h$-polynomial is $d$.  Hibi and Matsuda later refined this
result in \cite{HM2} to show that $I$ can be taken to be a lexsegment ideal, and later, with Van Tuyl  \cite{HMVT}, 
showed that $I$ could
be an edge ideal.  Further comparisons of the regularity and
degree have been carried out for the edge ideals of
Cameron-Walker graphs \cite{HKMVT} and binomial edge
ideals \cite{HM3,KK}.  In this note we
compare these invariants for the toric ideals of 
finite simple graphs.

Given a finite simple graph $G$ on the vertex
set $V = \{v_1,\ldots,v_n\}$ with edge set $E = \{e_1,\ldots,e_q\}$,
the {\it toric ideal of $G$}, denoted $I_G$, is the 
kernel of the map $\varphi : \K[E] = \K[e_1,\ldots,e_q] \rightarrow \K[v_1,\ldots,v_n]$ given by $\varphi(e_i) = v_{i_1}v_{i_2}$ where $e_i = \{v_{i_1},v_{i_2}\} \in E$.  
Some properties of the homological invariants of 
$I_G$ can be found in \cite{BOVT,FHKVT,GHKKVTP,GM,HBOK,HHKOK,OH,TT}.
Our main result adds to this list of properties, and contributes
to Hibi and Matsuda's program.

\begin{theorem}\label{maintheorem}
Let $4 \leq r \leq d$ be integers.  Then there is a connected finite
simple graph $G = G_{r,d}$ such that the toric ideal of $G$
satisfies ${\rm reg}(\K[E]/I_G) = r$ and $\deg h_{\K[E]/I_G}(x) =d$.  
\end{theorem}

\noindent
The proof of Theorem \ref{maintheorem} has two components.  First,
we consider the family of graphs constructed from the complete
bipartite graph $K_{2,t}$ by adjoining a ``triangle" to each
vertex of degree two (see Figure \ref{fig:graph G_5}).  
We prove that the toric ideals of the graphs in this family  have a unique extremal graded Betti number.  We use this
fact to show that
for any $e \geq 5$, we can construct a graph $G$ 
such that $\K[E]/I_G$ has 
regularity 4 and the degree of its $h$-polynomial
is $e$. 
The second component is to leverage the splitting techniques
of the authors and Hofscheier \cite{FHKVT} to
create the desired graphs of Theorem \ref{maintheorem} from the graphs in this family.  As a bonus corollary, we give a new proof
for the main result of \cite{HHKOK} which compared
the depth and dimension of toric ideals of graphs
(see Corollary~\ref{hibiresult}).

Our paper is structured as follows.
In Section \ref{s. background}, we give the relevant background, including
the undefined terms from the introduction. We also recall some
tools from \cite{FHKVT}; they are used to show that if $r \geq 1$,
there is a graph 
$G$ with $r = {\rm reg}(\K[E]/I_G) = \deg h_{\K[E]/I_G}(x)$. 
In Section \ref{s. reg 4 d 5}, we introduce a family of
connected graphs, and we show we can control the values
of ${\rm reg}(\K[E]/I_G)$ and $\deg h_{\K[E]/I_G}(x)$, where  $I_G$ is the toric ideal
of a graph in this family.
These graphs can then be
used to prove Theorem~\ref{maintheorem}.  We conclude
with remarks in Section 4 about pairs $(r,d)$ not covered by Theorem \ref{maintheorem}.

%%%%%%%%%%%%%%%%%%%%%%%%%%%%%%%%%%%%%%%%%%%%%%%%%%%%%

\section{Preliminaries}\label{s. background}

We recall the relevant background on
homological invariants and toric ideals of graphs.

\subsection{Homological invariants}
If $I$ is a homogeneous ideal of $R$, then 
the minimal graded free resolution of $R/I$ has the form
\[0 \rightarrow \bigoplus_{j \in \mathbb{N}} R(-j)^{\beta_{p,j}(R/I)}
\rightarrow \cdots \rightarrow \bigoplus_{j \in \mathbb{N}} R(-j)^{\beta_{1,j}(R/I)}
\rightarrow R \rightarrow R/I \rightarrow 0\]
where $R(-j)$ is the ring $R$ with its grading shifted by
$j$, and $\beta_{i,j}(R/I)=\dim_\K {\rm Tor}_i^R(R/I, \K)_j$ is called the ${i,j}$-th {\em graded Betti number} of $R/I$.  
The {\em (Castelnuovo-Mumford) regularity} of $R/I$ is 
\[{\rm reg}(R/I) = \max\{j-i ~|~ \beta_{i,j}(R/I) \neq 0\}.\]
The {\em projective dimension} of $R/I$ is the length of the 
minimal graded free resolution, that is
\[
\pdim(R/I) = \max\{i ~|~ \beta_{i,j}(R/I) \neq 0\} \text{.}
\]

The {\em Hilbert series} of a standard graded $\K$-algebra $R/I$ is the formal power series
\[
HS_{R/I}(x) = \sum_{i \geq 0} \left[\dim_\K (R/I)_i\right]x^i
\]
where $\dim_\K (R/I)_i$ is the dimension of $i$-th graded piece of $R/I$.
The Hilbert series of $R/I$ can be read from any resolution of $R/I$ 
(e.g., see \cite[p. 100]{hhGTM}). In particular, 
%if $B_j:=\left(\sum_{i}(-1)^i \beta_{i,j}(R/I)\right)$, then
\begin{equation}\label{eq. HS and Betti}
HS_{R/I}(x) = \dfrac{1+\sum_{i,j}(-1)^i \beta_{i,j}(R/I)x^{j}}{(1-x)^n}.
%= \frac{1+\sum_j B_j x^{j}}{(1-t)^n}.    
\end{equation}

By the Hilbert-Serre Theorem (e.g., see \cite[Theorem 5.1.4]{V}) there is a polynomial $h_{R/I}(x)\in \Z[x]$, called the {\em $h$-polynomial} of $R/I$, such that $HS_{R/I}$ can be written as
\begin{equation}\label{eq. reduced HS}
HS_{R/I}(x) = \frac{h_{R/I}(x)}{(1-x)^{\dim(R/I)}} ~~\mbox{with $h_{R/I}(1) \neq 0$} \text{,}    
\end{equation}
where $\dim(R/I)$ denotes the Krull dimension of $R/I$.

We recall a fact about extremal Betti numbers;  
see \cite{BCHP}
for more on their properties.

\begin{definition}
A graded Betti number of $R/I$, say $\beta_{a,b}(R/I)\neq 0$,  is {\em extremal} if $\beta_{i,j}(R/I)= 0$ for any pair $(i,j)$ such that $i\ge a$ and  $j>b$ and  $j-i\ge b-a$.
\end{definition}

\begin{lemma}\label{. reg deg 1 extremal betti}
Suppose $\beta_{a,b}(R/I)$ is the only extremal Betti
number of $R/I$.  Then   $\reg(R/I)=b-a$, $\pdim(R/I)=a$,
and  $\deg h_{R/I} (x)= b-\dim R+\dim R/I.$
\end{lemma}
 
 \begin{proof}
 Since  $\beta_{a,b}(R/I)$ is an extremal Betti number, from the definition, we have $\beta_{a,b}(R/I)\neq 0$ and $\beta_{i,j}(R/I)= 0$ for any $i\ge a$,  $j>b$ and  $j-i\ge b-a$.
 Moreover, because it is the unique extremal Betti number, $\beta_{i,j}(R/I)= 0$ 
if either $i\ge a $ or $j>b$ (otherwise there must be some other extremal Betti). 
Thus, the  Betti table of $R/I$ has a rectangular shape and the pair $(a,b)$ determines the regularity and the
projective dimension. Furthermore, from equation \eqref{eq. HS and Betti}, the degree of the non-reduced numerator
in the Hilbert series is $b$, so by  \eqref{eq. reduced HS}, the degree of the $h$-polynomial is $b-\dim R+\dim R/I$.
 \end{proof}

A monomial ideal $I \subseteq R$ is said to have  
{\em linear quotients} if its minimal generators $\{g_1,\ldots,g_m\}$ can
be ordered so that the quotient ideal
$\langle g_1,\ldots, g_{j-1}\rangle: \left\langle  g_j \right\rangle$
is generated by variables for every $j=2,\ldots,m$.
Linear quotients were first defined in \cite{HT}.
%among
%their nice properties,
%one can compute the graded Betti numbers of 
%monomial ideals with linear quotients as a function of 
%the number of generators of the quotient ideals.  Indeed,  by 
By {\cite[Corollary~2.7]{SV}}, a monomial ideal $I \subseteq R$ 
with linear quotients with respect to the ordering
$g_1,\ldots,g_m$,
has graded Betti numbers given by the formula
\begin{equation}\label{eq. betti from quotients}
\beta_{i+1,i+j}(R/I)=\sum_{1 \leq p \leq m,\text{ }{\rm  deg}(g_{p})=j} \binom{n_{p}}{i} ~~\mbox{for $i \geq 0$}
\end{equation}
where $n_p$ denotes the number of different variables
generating $\langle g_1,\ldots,g_{p-1}\rangle : \langle g_p \rangle$.

For a fixed monomial ordering, we let ${\rm in}(I)$ denote the 
{\it initial ideal
of $I$}.  It is well known that $\beta_{i,j}(R/I) \leq \beta_{i,j}(R/{\rm in}(I))$ for
all $i,j \geq 0$ (e.g., see \cite[Theorem 22.9]{PeevaBook}).  The following
result, found in \cite[Lemma 2.6]{GHKKVTP}, gives a criterion 
for when we have equality
for all $i,j \geq 0$.

\begin{lemma}\label{equalbettinumbers}
Fix a monomial order.  Suppose that $I \subseteq R$ is a homogeneous ideal
such that $\beta_{i,i+j}(R/I) = \beta_{i,i+j}(R/{\rm in}(I))$ for all $i$
and all $j \neq k$.  Then $\beta_{i,i+k}(R/I) = \beta_{i,i+k}(R/{\rm in}(I))$ for all~$i\geq0$.
\end{lemma}

\subsection{Toric ideals of graphs}
We now turn to toric ideals of graphs,
as defined in the introduction.  Note that 
if $G$ is a finite simple graph,
then the toric ideal  $I_{G}$ is a prime homogeneous 
binomial ideal.  Many of the algebraic and 
geometric invariants of $I_G$ 
depend upon the combinatorics of $G.$ In order to discuss these 
results, we briefly introduce some relevant terminology and results
(see Villareal \cite{V} and Herzog, Hibi, and Ohsgui \cite{Binomi} for details).
Note that if $G = (V,E)$ is a finite simple graph, we may
sometimes write $\K[E]$ for $\K[e ~|~ e \in E]$ and $\K[G]$ for the ring $\K[E]/I_G$.

If $G$ is a finite simple graph, a {\emph{walk}} in $G$ is a sequence of edges $w=(e_{1}, e_2, \ldots, e_{k})$ such that $e_{i}\cap e_{i+1}\ne \emptyset$  for $i=1,\ldots,k-1$. Equivalently, a walk is a sequence of vertices $(x_{1},\dots,x_{k},x_{k+1})$ such that $\{x_{i},x_{i+1}\}\in E$ for $i=1,\dots, k$. A walk is an {\emph{even walk}} if $k$ is even. A {\emph{closed walk}} is a walk where $x_{k+1}=x_1$. 
Two closed even walks $(e_0,\ldots,e_{2k-1})$ and
$(e'_0,\ldots,e'_{2k-1})$ are equivalent up
to a {\em circular permutation} if
there is an $i$ such that 
$e_j = e'_{j+i}$ for all $j$ where $j+i$ is taken 
modulo $2k$ (or if the walk is in the reverse order,
i.e., $e_j = e'_{(2k-j)+i}$ for all $j$).

A finite graph $G$ is {\em connected} if for every $x,y \in V$ with $x\neq y$, there exists a walk having $x$ as its first vertex and $y$ as its last.
A closed walk $(e_1,\ldots,e_k)$ where each vertex and edge is distinct is called
a {\it cycle} of length $k$.
A graph $G$ is {{\em bipartite}} if there are no odd cycles in $G$.   An {\em $n$-cycle}, denoted $C_n$, is the graph
with vertex set $V = \{x_1,\ldots,x_n\}$ and
edge set $E = \{\{x_1,x_2\},\{x_2,x_3\},\ldots,\{x_{n-1},x_n\},
\{x_n,x_1\}\}$.

The generators of the toric ideal $I_G$ can be obtained from closed even walks in $G$;  we sketch out this connection.  To each
closed even walk $w=(e_{i_{1}},e_{i_{2}},\dots, e_{i_{2n}})$ 
in $G$, we can associate the binomial $f_{w}$ defined by 
\[f_w=\prod_{2 \nmid j}e_{i_j} - \prod_{2 \mid j} e_{i_j}\in I_{G}.\]
Note that it is straightforward to verify that $\varphi(f_w) = 0$ where
$\varphi:\K[e_1,\ldots,e_q] \rightarrow \K[v_1,\ldots,v_n]$ is the 
map defining $I_G = {\rm ker}(\varphi)$.   Among all closed
even walks, we identify a special subset.

\begin{definition} A binomial $f_1-f_2 \in I_G$ is 
{\emph{primitive}} if there exists no binomial $g_1-g_2 \in I_G$ such that $g_1 \mid f_1$ and $g_2 \mid f_2$. A closed even walk $w$ in a graph $G$ is said to be {\emph{primitive}} if the corresponding binomial $f_w$ is primitive in $I_G$.
\end{definition} 

The importance of primitive closed even walks lies in the next
theorem.

\begin{theorem}[{\cite[Proposition 10.1.10]{V}}]{\label{UGB}} The set of binomials associated with primitive closed even walks is a universal Gr\"obner basis  of $I_{G}$. \end{theorem}

We round out this section by specializing one of the 
results of
\cite{FHKVT} that will be a  key ingredient in our proof
of Theorem \ref{maintheorem}.  
Recall that given a graph $G = (V,E)$ and $W \subseteq V$,  the \emph{induced subgraph} of $G$ on $W$ is the graph with vertex set $W$ and edge set $\{e \in E ~|~ e \subseteq W\}$.   Following
\cite[Construction 4.1]{FHKVT},  
let $G_1, G_2$ be two graphs and suppose that $H_1 \subseteq G_1, H_2\subseteq G_2$ are two induced subgraphs which are isomorphic with respect to some graph isomorphism   $\varphi : H_1 \to H_2$. We define the \emph{glued graph} $G_1 \cup_\varphi G_2$ of $G_1$ and $G_2$ along $\varphi$ as the disjoint union of $G_1$ and $G_2$, and we use $\varphi$ to identify vertices and edges in $H_{1}$ with their images in $H_{2}$. At times, we may be more informal and say that  \emph{$G_1$ and $G_2$ is glued along $H$} if the induced subgraphs $H \cong H_1$ and $H \cong H_2$ and isomorphism $\varphi$ are clear. 

It was shown in \cite{FHKVT} that under some hypotheses on
$G_1$ and $G_2$, if the $G_1$ and $G_2$ are glued along
some induced subgraph $H$, then many of the homological
invariants of $G_1 \cup_\varphi G_2$ are related to those of $G_1$ and
$G_2$.  In particular, if we specialize 
\cite[Corollary 3.11]{FHKVT}, we have 
the following result.

\begin{theorem}\label{t. 3.9 FHKVT}
	Let $G$ be any finite simple connected graph, and let
	$C_{2s}$ be an even cycle of length $2s \geq 4$.  Let $e$ be 
	any edge of $G$ and let $e'$ be any edge of $C_{2s}$.  If	$G' = (V',E')$ is the graph obtained by gluing $G$ and $C_{2s}$ along	$e \cong e'$,  then
	\begin{enumerate}
	    \item[$(i)$]  ${\rm reg}(\K[G']) = {\rm reg}(\K[G])+s-1$, and
	    \item[$(ii)$] $\deg h_{\K[G']}(x) = \deg h_{\K[G]}(x) + s-1$.
	\end{enumerate}
	\end{theorem}

\begin{corollary}\label{p. inductive step reg and degree}
Let $G=(V,E)$ be a connected graph with 
$\deg h_{\K[G]}(x)=d$ and $\reg(\K[G])=r$.
Then there exists a connected graph $G'=(V',E')$ with $\deg h_{\K[G']}(x)=d+1$ and $\reg(\K[G'])= r+1.$
\end{corollary}

\begin{proof}
	By Theorem \ref{t. 3.9 FHKVT}, if we glue a $C_4$
	along any edge of $G$, we get the desired result.
\end{proof}

For all integers
$1 \leq r$, there is a graph $G$ satisfying  $\deg h_{\K[G]}(x) 
= {\rm reg}(\K[G]) = r$.

\begin{example}\label{e. deg=reg}
Consider the graph $C_4^{(r)}$, where $r\ge 1$ is an integer, on the vertex set $V^{(r)}=\{x_1, \ldots, x_{2r+2}\}$
and edge set
$E^{(r)}=\{\{x_1,x_2\}\}\cup\{ \{x_1, x_{2i+1}\}, \{x_2, x_{2i+2}\},   \{x_{2i+1},  x_{2i+2}\}\ |\ iv=1,\ldots,r \}.$
So, the graph $C_4^{(r)}$ consists of $r$ squares glued along one edge. Since, $C_4=C_4^{(1)}$ has $\deg h_{\K[C_4]}(x) = 
{\rm reg}(\K[C_4]) = 1$ then, iteratively from Corollary \ref{p. inductive step reg and degree}, we get  $\deg h_{\K[C_4^{(r)}]}(x) = 
{\rm reg}(\K[C_4^{(r)}]) = r$.   
\end{example}
%%%%%%%%%%%%%%%%%%%%%%%%%%%%%%%%%%%%%%%%%%%%%%%%%%%%%%%%%%%%%%%%%%%%%%%%%%%%%%%%%%%%%%%%%%%%%%%%%%%%%%%%%%%%%%%%%%%%%%%%%%%%%%%%%%%%%%%%%%%%%%%%%%%%%%%%%%%%%%%%%%%%%%%%%%%%%%%%%%%%%%%%%%%%%%%%%%%%%%%%%%%%%%%%%%%%

\section{Some homological invariants of the toric ideal for a fixed family of graphs}\label{s. reg 4 d 5}

In this section we construct a family of simple graphs $G_{t}$ with $t\ge 2$ such that ${\rm reg}(\K[G_t]) = 4$ and 
$\deg h_{\K[G_t]}(x) = t+3$.  
By combining this family with Corollary \ref{p. inductive step reg and degree}, we can prove Theorem \ref{maintheorem}.

To help the reader, we sketch out the broad strokes that we take
in this section.  We begin by defining a graph $G_t$ on 
$t+6$ vertices and $2t+6$ edges, where $t\ge2$ is an integer.
We then describe a set $\mathcal{G}$ of binomials that form  a universal Gr\"obner basis for $I_{G_t}$
 and a set $\mathcal{M}$ of minimal generators of $\init(I_{G_t}),$ the initial ideal of $I_{G_t}$ for a given monomial ordering. We show that $\init(I_{G_t})$ has linear quotients. Lastly, we prove that all the graded Betti numbers of $\K[{G_t}]$ coincide with the ones of $\K[E_t]/\init(I_{G_t}),$ and that there exists a unique extremal Betti number. We derive Theorem \ref{maintheorem} from these facts. 
 
 We begin by formally defining the graphs of interest.
\begin{definition}\label{d.graph G_t} Let $t \geq 2$ be an integer. The graph $G_{t}$ is defined having the vertex and edge sets:
\[
V_t=\{x_{1},x_{2},y_{1},\ldots,y_{t},z_{1},z_{2},w_{1},w_{2}\}, ~\mbox{and}
\]
\[
E_t=\{ \{x_{i},y_{j}\} \mid 1 \leq i \leq 2, 1 \leq j \leq t \} \cup \{ \{x_{1},z_{1}\}, \{z_{1},z_{2}\}, \{z_{2},x_{1}\} \}\cup \{ \{x_{2},w_{1}\}, \{w_{1},w_{2}\}, \{w_{2},x_{2}\}\}.
\]
We label the edges of $G_{t}$ as follows:  $e_{1}=\{x_{1},z_{1}\}$, $e_{2}=\{z_{1},z_{2}\}$, $e_{3}=\{z_{2},x_{1}\}$, $f_{1}=\{x_{2},w_{1}\}$, $f_{2}=\{w_{1},w_{2}\}$, $f_{3}=\{w_{2},x_{2}\}$ and, for $i\in{\{1,\ldots,t\}}$, $a_{i}=\{x_{1},y_{i}\}$ and $b_{i}=\{x_{2},y_{i}\}$.
\end{definition}
Note that the subgraph of $G_t$ on the vertices $\{x_1,x_2, y_1, \ldots, y_t\}$ is a complete bipartite graph $K_{2,t}$ consisting of only the edges $\{a_1,\ldots,a_t,b_1,\ldots,b_t\}$.  Thus, less formally, the graph $G_t$ is obtained from the complete bipartite graph $K_{2,t}$ by joining a 3-cycle to
each of the two vertices of degree $t$.   See Figure \ref{fig:graph G_5} for the case
$t=5$.  Note that the toric ideals of these graphs were also considered in \cite{HHKOK}.

\begin{figure}[ht]
    \centering
\begin{tikzpicture}
 [scale=.45,auto=left,every node/.style={circle,fill=lightgray,
 scale=.5, minimum size=3.2em}]
\node (x1) at (3,6) {\text{$x_{1}$}};
\node (x2) at (9,6) {\text{$x_{2}$}};
\node (y1) at (0,0) {\text{$y_{1}$}};
\node (y2) at (3,0) {\text{$y_{2}$}};
\node (y3) at (6,0) {\text{$y_{3}$}};
\node (y4) at (9,0) {\text{$y_{4}$}};
\node (y5) at (12,0) {\text{$y_{5}$}};
\node (z1) at (1,9) {\text{$z_{1}$}};
\node (z2) at (5,9) {\text{$z_{2}$}};
\node (w1) at (7,9) {\text{$w_{1}$}};
\node (w2) at (11,9) {\text{$w_{2}$}};

\node[fill=white] (a1) at (1.6,4) {\text{$a_{1}$}};
\node[fill=white] (a2) at (2.7,4) {\text{$a_{2}$}};
\node[fill=white] (a3) at (3.7,4) {\text{$a_{3}$}};
\node[fill=white] (a4) at (4.7,4) {\text{$a_{4}$}};
\node[fill=white] (a5) at (5.2,4.8) {\text{$a_{5}$}};
\node[fill=white] (b1) at (4,2.3) {\text{$b_{1}$}};
\node[fill=white] (b2) at (5.8,2.3) {\text{$b_{2}$}};
\node[fill=white] (b3) at (7.5,2.3) {\text{$b_{3}$}};
\node[fill=white] (b4) at (9.3,2.8) {\text{$b_{4}$}};
\node[fill=white] (b5) at (10.9,2.8) {\text{$b_{5}$}};
\node[fill=white] (e1) at (1.7,7.4) {\text{$e_{1}$}};
\node[fill=white] (e2) at (3,9.3) {\text{$e_{2}$}};
\node[fill=white] (e3) at (4.3,7.4) {\text{$e_{3}$}};
\node[fill=white] (f1) at (7.7,7.4) {\text{$f_{1}$}};
\node[fill=white] (f2) at (9,9.3) {\text{$f_{2}$}};
\node[fill=white] (f3) at (10.3,7.4) {\text{$f_{3}$}};

\draw (z1) -- (z2);
\draw (z2) -- (x1);

\draw (w1) -- (w2);
\draw (w2) -- (x2);

\draw (x1) -- (z1);
\draw (x1) -- (y1);
\draw (x1) -- (y2);
\draw (x1) -- (y3);
\draw (x1) -- (y4);
\draw (x1) -- (y5);

\draw (x2) -- (w1);
\draw (x2) -- (y1);
\draw (x2) -- (y2);
\draw (x2) -- (y3);
\draw (x2) -- (y4);
\draw (x2) -- (y5);
\end{tikzpicture}
    \caption{The graph $G_5$.}
    \label{fig:graph G_5}
\end{figure}
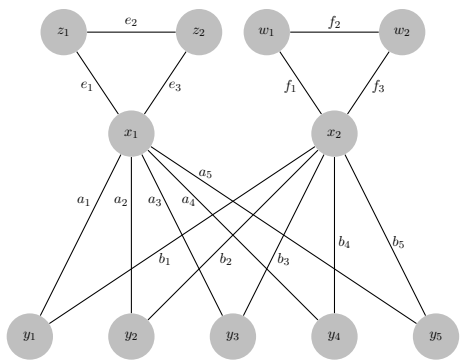

Going forward, we work in the standard graded polynomial ring $$\K[E_t] = \K[a_{1},\dots,a_{t},f_{1},f_{2},f_{3},e_{1},e_{2},e_{3},b_{1},\dots,b_{t}].$$ 
Let $>$ denote the graded reverse lexicographic monomial ordering on $\K[E_t]$  satisfying 
\begin{equation}\label{monomialorder}
a_{1}>\cdots >a_{t}>f_{1}>f_{2}>f_{3}>e_{1}>e_{2}>e_{3}>b_{1}>\cdots >b_{t}.
\end{equation} We denote the {\em initial ideal} of an ideal
$I$ with respect to this ordering by ${\rm in}(I)$.

Before focusing on $I_{G_t}$, we summarize some known results about the toric ideal of $I_{K_{2,t}}$. Here, we see  $K_{2,t}$, the complete bipartite graph, as the induced subgraph of $G_t$ on  $V':=\{x_1,x_2, y_1,\ldots, y_t\}$. 

\begin{lemma}\label{completebipartite}
Fix some integer $t \geq 2$.
Using the same labelling as in Definition \ref{d.graph G_t}, let $I_{K_{2,t}}$ be the toric ideal of the graph $K_{2,t}=(V',E')$  in the 
polynomial ring $\K[E']=\K[a_1,\ldots,a_t,b_1,\ldots,b_t]$.  Then
\begin{enumerate}
    \item[$(i)$] $I_{K_{2,t}} = \langle a_{i}b_{j}-a_{j}b_{i} ~|~ 1 \leq i < j \leq t \rangle$;
    \item[$(ii)$] ${\rm in}(I_{K_{2,t}}) = \langle a_{i}b_{j}\ |\   1 \leq j < i \leq t \rangle$ with respect to the graded reverse lexicographical order where
    $a_1 > a_2 > \cdots >a_t> b_1 > \cdots > b_t$;
    \item[$(iii)$] ${\rm in}(I_{K_{2,t}})$ has linear quotients if one
    orders the generators with respect to the graded reverse lexicographical order; and
    \item[$(iv)$] if $\{g_1,\ldots,g_k\}$ are the generators of ${\rm in}(I_{K_{2,t}})$
    ordered with respect to the graded reverse lexicographical order,
    then $n_p \leq t-1$ for all $p$, where $n_p$ is the number 
    of generators of $\langle g_1,\ldots,g_{p-1}\rangle:\langle g_p \rangle$
    for $p=1,\ldots,k$;
    \item[$(v)$]  $\beta_{i,j}(\K[K_{2,t}]) = \beta_{i,j}(\K[E']/{\rm in}(I_{K_{2,t}}))$
    for all $i,j \geq 0$.%
\end{enumerate}
\end{lemma}

\begin{proof}
Statements
$(i)$ and $(ii)$ follow from \cite[Remark 3.4]{GHKKVTP} which shows that
the given generators are a universal Gr\"obner basis of $I_{K_{2,t}}$.
Statements $(iii), (iv)$, and $(v)$ follow from \cite[Corollary 2.8]{GHKKVTP}.
\end{proof}

The next result describes the set of primitive binomials of $I_{G_{t}}$ which we denote by $\mathcal{G}$.

\begin{theorem}\label{l. generators toric ideal G_t}
For any integer $t\ge 2$, the ideal $I_{G_t}$ is generated by the primitive binomials  in $\mathcal G=\mathcal G_1\cup \mathcal G_2\cup \mathcal G_3$ where
\begin{enumerate}
\item[$(i)$] $\mathcal G_1= \{a_{i}b_{j}-a_{j}b_{i}\ |\ 1 \leq i < j \leq t\}$, 
\item[$(ii)$] $\mathcal G_2= \{a_{i}a_{j}f_{1}f_{3}e_{2}-f_{2}e_{1}e_{3}b_{i}b_{j}\ | \ 1 \leq i < j \leq t\}$, and
\item[$(iii)$] $\mathcal G_3=\{ a_{i}^{2}f_{1}f_{3}e_{2}-f_{2}e_{1}e_{3}b_{i}^{2}\ |\ 1 \leq i \leq t\}$.
\end{enumerate}
In particular, $\mathcal{G}$ is a universal
Gr\"obner basis for $I_{G_t}$.
\end{theorem}

\begin{proof} 
By Theorem \ref{UGB}, it suffices to show that the binomials in $\mathcal{G}$ correspond to
the primitive closed even walks in $G_t$. We only need to identify these even walks up to 
a circular permutation, since the associated binomials
will be equal up to a sign.

Note that the elements of 
$\mathcal{G}$ correspond to the following
closed even walks in the graph $G_t$:
\begin{itemize}
    \item $(a_i,b_i,b_j,a_j)$, where $1 \leq i < j \leq t$, 
    \item $(a_i,b_i, f_1, f_2, f_3,b_j,a_j, e_1, e_2, e_3 )$, where $1 \leq i < j \leq t$, and
    \item $(a_i,b_i, f_1, f_2, f_3,b_i,a_i, e_1, e_2, e_3 )$  where $1 \leq i \leq t$.  
 \end{itemize}   
 However, as noted in \cite{HHKOK} (prior to Lemma 2.1), these
 closed even walks form a complete set of primitive
 closed even walks.
\end{proof}

\begin{corollary}\label{mono}
Using the graded reverse lexicographic order that satisfies \eqref{monomialorder}, 
we have that ${\rm in}(I_{G_{t}})$ is generated by the monomials in $\mathcal{M}=\mathcal{M}_{1}\cup\mathcal{M}_{2}\cup\mathcal{M}_{3}$  where:
\begin{enumerate}
\item[$(i)$] $\mathcal M_1= \left\{
     a_{i}b_{j}\ |\   1 \leq j < i \leq t \right\}$,
\item[$(ii)$] $\mathcal M_2= \left\{
a_{i}a_{j}f_{1}f_{3}e_{2} \ |\   1 \leq i < j \leq t \right\}$, and 
\item[$(iii)$] $\mathcal M_3=\left\{ a_{i}^{2}f_{1}f_{3}e_{2}\ | \ 1 \leq i  \leq t \right\}$.
\end{enumerate}
Furthermore, $\mathcal{M}$ is a minimal set of generators for ${\rm in}(I_{G_{t}})$.
\end{corollary}

\begin{proof}
That $\mathcal{M}$ is a generating set with respect to the given order follows from Theorem \ref{l. generators toric ideal G_t}. That it is minimal follows from the fact that none of the monomials are divided by any of the others.
\end{proof}

We will show ${\rm in}(I_{G_t})$ has  linear quotients with respect to an order of its generators.

\begin{theorem}\label{p. linear quotient G_t}
Let $\mathcal{M}_1, \mathcal{M}_2$ and $\mathcal{M}_3$ be as in Corollary
\ref{mono}, and order each set from smallest to largest with
respect to the graded reverse lexicographical order.  
Then the initial ideal of $I_{G_t}$
 \begin{align*}
{\rm in}(I_{\text{G}_{t}})=\left\langle\right. & a_{t}b_{t-1},\  a_{t}b_{t-2},\dots,  a_{t-1}b_{t-2}, \dots,  a_{2}b_{1}, \\
& a_{t}a_{t-1}f_{1}f_{3}e_{2},\ a_{t}a_{t-2}f_{1}f_{3}e_{2}, \dots, a_{2}a_{1}f_{1}f_{3}e_{2}, \\
& a_{t}^{2}f_{1}f_{3}e_{2},\ a_{t-1}^{2}f_{1}f_{3}e_{2}, \dots, a_{1}^{2}f_{1}f_{3}e_{2}\left.\right\rangle
\end{align*}
has linear quotients with respect to this order of the generators.
Furthermore, if ${\rm in}(I_{G_t}) = \langle g_1,\ldots,g_{t^2} \rangle$,
and $n_p$ is the number of generator of $\langle g_1,\ldots,g_{p-1} \rangle:
\langle g_p \rangle$, then 
$$\max\{n_p ~|~ 2 \leq p \leq t^2\} = 2t-2.$$
\end{theorem}

\begin{proof}
It follows from Corollary \ref{mono} that ${\rm in}(I_{G_t})$ has
$t^2$ generators.   Let $g_1,\ldots,g_{t^2}$ be these generators, ordered
as in the statement of the theorem.  For each $p \in \{2,\ldots,t^2\}$, let
$I(p) = \left\langle g_{1},\ldots,g_{p-1}\right\rangle : 
\left\langle  g_{p} \right\rangle$. 
A generating set of $I(p)$ is given by:
\begin{equation}\label{quotientformula}
I(p) = \left\langle \frac{LCM(g_1,g_p)}{g_p}, \frac{LCM(g_2,g_p)}{g_p}, \ldots, \frac{LCM(g_{p-1},g_p)}{g_p}\right\rangle.
\end{equation}

We first observe that the first $\frac{t(t-1)}{2}$ generators of
${\rm in}(I_{G_t})$ with respect to our ordering are the exact same as the
generators of ${\rm in}(I_{K_{2,t}})$ by Lemma \ref{completebipartite} $(ii)$. 
So, by Lemma \ref{completebipartite} $(iii)$, since this order has linear
quotients, $I(p)$ is generated by variables for $p=2,\ldots,\frac{t(t-1)}{2}$.

It suffices
to show that $I(p)$ is generated by variables for 
$p \in \left\{\frac{t(t-1)}{2}+1,\ldots,
t^2\right\}$.  We consider two cases.

\noindent   
{\bf Case 1.} Suppose that $g_{p}= a_{i}a_{j}f_{1}f_{3}e_{2}$ with 
$t \geq i>j \geq 1$.  Then the ideal $I(p)$ is 
 \[=
 \begin{cases}
 \left\langle a_{t}b_{t-1}, a_{t}b_{t-2}, \dots, a_{2}b_{1} \right\rangle:
 \langle a_{t}a_{t-1}f_1f_3e_2 \rangle & \mbox{if $i=t$ and $j=t-1$} \\
 \left\langle a_{t}b_{t-1}, a_{t}b_{t-2}, \dots, a_{2}b_{1},a_ta_{t-1}f_1f_3e_2,
 \ldots,a_{j+2}a_{j+1}f_1f_3e_2 \right\rangle: \langle a_{t}a_{j}f_1f_3e_2 \rangle
 & \mbox{if $i=t$ and $1 \leq j < t-1$} \\
 \left\langle a_{t}b_{t-1}, a_{t}b_{t-2}, \dots, a_{2}b_{1},a_ta_{t-1}f_1f_3e_2,
 \ldots,a_{i+1}a_{j}f_1f_3e_2\right\rangle: \langle a_{i}a_{j}f_1f_3e_2 \rangle
 & \mbox{if $i<t$}.
 \end{cases}
 \]
 If we calculate each ideal using \eqref{quotientformula}, we get
 \[
 I(p) =
 \begin{cases}
 \left\langle b_{1},\ldots,b_{t-1}  \right\rangle & \mbox{if $i=t$ and $j=t-1$} \\
 \left\langle b_{1},\ldots,b_{t-1},a_{j+1},\ldots,a_{t-1} \right\rangle
 & \mbox{if $i=t$ and $1 \leq j < t-1$} \\
 \left\langle b_1,\ldots,b_{i-1},a_{j+1},\ldots,a_{i-1},a_{i+1},\ldots, a_t \right\rangle
 & \mbox{if $i<t$.}
 \end{cases}
 \]
  
\noindent
{\bf Case 2.} If $g_p =a_i^2f_1f_3e_2$, then
\[
I(p) = 
\begin{cases}
\langle a_tb_{t-1},\ldots,a_2a_1f_1f_3e_2 \rangle:\langle a_t^2f_1f_3e_2 \rangle
& \mbox{if $i=t$} \\
\langle a_tb_{t-1},\ldots,a_2a_1f_1f_3e_2,a_1^2f_1f_3e_2,\ldots,
a_{i+1}^2f_1f_3e_2\rangle:\langle a_i^2f_1f_3e_2 \rangle
& \mbox{if $1 \leq i < t$.} \\
\end{cases}
\]
Computing each colon ideal gives
\[
I(p) = 
\begin{cases}
\langle b_1,\ldots,b_{t-1},a_1,\ldots,a_{t-1}\rangle
& \mbox{if $i=t$} \\
\langle b_1,\ldots,b_{i-1},a_1,\ldots,a_{i-1},a_{i+1},\ldots,a_t \rangle
& \mbox{if $1 \leq i < t$.} \\
\end{cases}
\]
It thus follows that ${\rm in}(I_{G_t})$ has linear quotients with respect to the given order.

To prove the final statement, it follows that $n_p \leq t-1$ for $p=2,\ldots,\frac{t(t-1)}{2}$ by Lemma \ref{completebipartite} $(iv)$.
On the other hand, from our above computations, we saw that
\[\langle a_tb_{t-1},\ldots,a_2a_1f_1f_3e_2 \rangle:\langle a_t^2f_1f_3e_2 \rangle =
\langle b_1,\ldots,b_{t-1},a_1,\ldots,a_{t-1}\rangle\]
has $2t-2$ generators, and every ideal $I(p)$ 
with $\frac{t(t-1)}{2}+1 \leq p \leq  t^2$ has 
$n_p \leq 2t-2$.
\end{proof}

\begin{corollary}\label{samebettinumbers}
For any integer $t \geq 2$, we have 
$\beta_{i,i+j}(\K[G_t]) = \beta_{i,i+j}(\K[E_t]/{\rm in}(I_{G_t})) ~~\mbox{for all
$i,j \geq 0$}.$
\end{corollary}

\begin{proof}
Recall that we have $\beta_{i,i+j}(\K[G_t]) \leq 
\beta_{i,i+j}(\K[E_t]/{\rm in}(I_{G_t}))$ for all $i,j \geq 0$.  
Because ${\rm in}(I_{G_t})$ has linear
quotients and is only generated in degrees $2$ and $5$,
formula \eqref{eq. betti from quotients} thus gives 
\[\beta_{i,i+j}(\K[G_t]) = \beta_{i,i+j}(\K[E_t]/{\rm in}(I_{G_t})) = 0 ~~\mbox{for all
$i \geq 0$ and all
$j \neq 1,4$.}\]

On the other hand, the generators of $I_{G_t}$ of degree two
are the exact same as the generators of~$I_{K_{2,t}}$ by
Theorem \ref{l. generators toric ideal G_t} and Lemma \ref{completebipartite}.
So $\beta_{i,i+1}(\K[G_t]) = \beta_{i,i+1}(\K[K_{2,t}])$ for all $i \geq 0$. 
The minimal generators of ${\rm in}(I_{G_t})$ 
of degree $2$ are also the minimal generators of ${\rm in}(I_{K_{2,t}})$.  So
\[\beta_{i,i+1}(\K[K_{2,t}]) = \beta_{i,i+1}(\K[G_t])
\leq \beta_{i,i+1}(\K[E_t]/{\rm in}(I_{G_t})) = \beta_{i,i+1}(\K[E_t]/{\rm in}(I_{K_{2,t}})) 
= \beta_{i,i+1}(\K[{K_{2,t}}])\]
where the last inequality is Lemma \ref{completebipartite} $(v)$.
So we have shown that $\beta_{i,i+j}(\K[G_t]) = \beta_{i,i+j}(R/{\rm in}(I_{G_t}))$
for all $i, j \geq 0$ except $j=4$. To complete the proof, we now apply Lemma 
\ref{equalbettinumbers}.
\end{proof}

\begin{remark}
It is possible to find an explicit formula
for $\beta_{i,i+j}(\K[E_t]/{\rm in}(I_{G_i}))$ using
the formula~\eqref{eq. betti from quotients},
and determining the exact values of $n_p$ for 
each $p$.  These
values can be
extracted from the proof of Theorem \ref{p. linear quotient G_t} and \cite[Theorem 3.6]{GHKKVTP}. 
\end{remark}
\begin{corollary}\label{extremalbetti}
For any integer $t \geq 2$,
$\beta_{2t-1,2t+3}(\K[G_t])$ is the  unique extremal Betti number of~$\K[{G_t}]$.
\end{corollary}
\begin{proof} 
By Corollary \ref{samebettinumbers}, it suffices to
show that $\beta_{2t-1,2t+3}(\K[E_t]/{\rm in}(I_{G_t}))$ is the unique extremal graded Betti number of $\K[E_t]/{\rm in}(I_{G_t})$.    Since the ideal ${\rm in}(I_{G_t})$ is generated in
degrees two and five, and because this ideal has linear quotients, any
extremal Betti number will have the form $\beta_{i,i+1}(\K[E_t]/{\rm in}(I_{G_t}))$ or $\beta_{i,i+4}(\K[E_t]/{\rm in}(I_{G_t}))$, where $i\ge 1$.  By Lemma \ref{completebipartite} $(iv)$ and 
formula \eqref{eq. betti from quotients} 
$\beta_{i,i+1}(\K[E_t]/{\rm in}(I_{G_t}))=0$ if $i \geq t$ since 
$n_p \leq t-1$ in this range.  On the other hand, since $n_p = 2t-2$
for a generator of degree five (by Theorem \ref{p. linear quotient G_t}),
and because this is the maximal such value for $n_p$, we have
$\beta_{2t-1,2t+3}(\K[E_t]/{\rm in}(I_{G_t})) \neq 0$ but $\beta_{i,i+4}(\K[E_t]/{\rm in}(I_{G_t})) =
0$ for all $i \geq 2t-1$.  Since $2t-1 \geq t-1$ because $t \geq 2$,
$\beta_{2t-1,2t+3}(\K[E_t]/{\rm in}(I_{G_t}))$ is the unique extremal graded Betti number.
% Recall that $\beta_{i+1,i+j}(R/I) = \beta_{i,i+j}(I)$ for $i \geq 0$. 
%By Corollary \ref{samebettinumbers}, it suffices to show that $\beta_{2t-1,2t+3}(\K[E_t]/{\rm in}(I_{G_t})) = \beta_{2t-2,2t+3}({\rm in}(I_{G_t}))$ is the unique extremal graded Betti number of ${\rm in}(I_{G_t})$.    Since this ideal is generated in degrees two and five, and because this ideal has linear quotients, any extremal Betti number will have the form $\beta_{i,i+2}({\rm in}(I_{G_t}))$ or $\beta_{i,i+5}({\rm in}(I_{G_t}))$.  By Lemma \ref{completebipartite} $(iv)$ and formula \eqref{eq. betti from quotients} $\beta_{i,i+2}({\rm in}(I_{G_t})) = 0$ if $i \geq t$ since $n_p \leq t-1$ in this range.  On the other hand, since $n_p = 2t-2$ for a generator of degree five (by Theorem \ref{p. linear quotient G_t}), and because this is maximal such value for $n_p$, we have $\beta_{2t-2,2t+3}({\rm in}(I_{G_t})) \neq 0$ but $\beta_{i,i+5}({\rm in}(I_{G_t})) = 0$ for all $i \geq 2t-1$.  Since $2t-2 \geq t-1$ because $t \geq 2$, $\beta_{2t-2,2t+3}({\rm in}(I_{G_t}))$ is the unique extremal graded Betti number.
\end{proof}

We can now compute the regularity and the degree of the $h$-polynomial of $\K[{G_t}]$.
\begin{theorem}\label{t. reg and deg G_t}
For any integer $t \geq 2$, the graph $G_t$ 
has $\reg (\K[G_t])= 4$ and $\deg h_{\K[G_t]}(x)=t+3$.
\end{theorem}
\begin{proof}
This results follows  by combining Corollary \ref{extremalbetti}
and Lemma \ref{. reg deg 1 extremal betti},
and using the fact that
$\dim \K[E_t] = 2t+6$ and $\dim(\K[G_t]) = |V(G_t)| = 2t+4$; see \cite[Corollary 10.1.21]{V} for the latter assertion.
\end{proof}

We now have all the pieces to prove the main theorem of this paper.
\begin{proof}[Proof of Theorem \ref{maintheorem}]
Let $(r,d)$ be a pair of integers such that $4\le r \le d$. 
If $r=d$, then the graph $C_4^{(r)}$ introduced in Example \ref{e. deg=reg} has the required invariants.
Assume now $r<d$.  Set $q =d-r+1\ge 2$.
By Theorem \ref{t. reg and deg G_t} the
graph $G_q$ has $\reg(\K[G_q])=4$ and $\deg h_{\K[G_q]}(x)=q+3=d-r+4.$ Thus, applying Corollary \ref{p. inductive step reg and degree}  $(r-4)$ times, we get the existence of a graph $G_{r,d}$ with  $\reg (\K[G_{r,d}])=r$ and $\deg h_{\K[G_{r,d}]}(x)=d.$ (The graph $G_{r,d}$ is obtained by gluing $G_q$ with $r-4$ squares $C_4$ along one edge, no matter which one.)
\end{proof}

As another consequence of
Corollary \ref{extremalbetti} we derive a new proof
of the main result of \cite{HHKOK}.

\begin{corollary}[{\cite[Theorem 0.2]{HHKOK}}]\label{hibiresult}
Fix integers $7 \leq f \leq d$.  Then there exists a graph
$G$ whose toric ring satisfies 
${\rm depth}(\K[G]) = f $ and $\dim (\K[G]) = d$.
\end{corollary}

\begin{proof}
As described in \cite{HHKOK}, the proof of the above
result hinges upon finding a graph on $k+6$ vertices with $k \geq 1$
whose toric ideal $I_G$ has the property that 
${\rm depth}(\K[G])=7$.
Using our notation, \cite{HHKOK} show that the graphs
$G_t$ with $t \geq 1$ (where $G_1$ is the graph of two triangles joined by a path of length two)
satisfy $\depth (\K[G_t]) = 7$.  But, for $t \geq 2$
this also follows from Lemma \ref{. reg deg 1 extremal betti}, Corollary \ref{extremalbetti}, and the Auslander-Buchsbaum formula since
$\depth(\K[G_t]) = \dim (\K[E_t]) - {\rm pdim}(\K[G_t]) = 
2t+6 - (2t-1) = 7$.  When $t=1$,
then $I_{G_t}$ has a single generator, so
${\rm pdim}(\K[G_t]) = 1$ and 
${\rm depth}(\K[G_1]) =8-1 = 7$.
The proof now runs as in the introduction of~\cite{HHKOK}.
\end{proof}

%%%%%%%%%%%%%%%%%%%%%%%%%%%%%%%%%%%%%%%%%%%%%%%%%%%%%%%%%%%%%%%%%%%%%%%%%%%%%%%%%%%%%%%%%%%%%%%%%%%%%%%%%%%%%%%%%%%%%%%%%%%%%%%%%%%%%%%%%%%%%%%%%%%%%%%%%%%%%%%%%%%%%%%%%%%%%%

\section{Further comments and observations}\label{s. other cases}

We now turn our attention to integers $d,r \geq 1$ not
covered by Theorem \ref{maintheorem}.   

While
Hibi and Matsuda \cite{HM1} showed that
for all $d,r \geq 1$, there is a monomial
ideal $I$ with $(r,d) = ({\rm reg}(R/I),\deg h_{R/I}(x))$,
this behaviour will not hold for toric ideals
of graphs.  In particular, if $r=1$, then 
$d$ must also equal $1$.

\begin{theorem}\label{restriction}
Let $G$ be a graph such that $\reg(\K[G])=1$. Then $\deg h_{\K[G]}(x)=1$. 
\end{theorem}
\begin{proof}
It can be assumed that the graph $G=(V,E)$ is connected.
Since $\reg(\K[G])=1$, then $\K[G]$ has a linear resolution (hence it has a unique extremal Betti number $\beta_{a,a+1}(\K[G])$) and, in particular, $I_G$ is only generated by quadratic binomials. Thus, from \cite[Corollary 5.26]{Binomi}, the ring $\K[G]$ is Cohen-Macaulay ($\depth(\K[G]) = \dim \K[G] $). So, the Auslander-Buchsbaum formula implies
$|E|-\dim \K[G] =  \pdim(\K[G]) = a$  
and, from Lemma \ref{. reg deg 1 extremal betti}, we get
$\deg h_{\K[G]}(x)=\reg(\K[G])$. 
\end{proof}

\begin{remark}
In the proof of Lemma \ref{restriction}, we saw that $\K[G]$ 
was Cohen-Macaualay, from which we deduced that
$\deg h_{\K[G]}(x) = {\rm reg}(\K[G])$.
As shown in \cite[Corollary B.4.1]{Vas}, this holds
in general, i.e., if $\K[G]$ is Cohen-Macaulay,
then the regularity and the degree of the $h$-polynomial
are equal.   We know of no example of a graph $G$ such that $I_G$ is generated in degrees $\le 3$ and $\K[G]$  is not a Cohen-Macaulay ring, thus suggesting
that if ${\rm reg}(\K[G]) \leq 3$, there may be
restrictions for $\deg h_{\K[G]}(x)$.  On the other hand,
the graph $G$ with vertex set $V = \{x_1,\ldots,x_8\}$ and
edge set 
\begin{eqnarray*}
E &=& \{\{x_1,x_2\},\{x_1,x_3\},\{x_2,x_3\},\{x_3,x_4\},
\{x_4,x_5\},\{x_4,x_6\},\{x_4,x_7\},\{x_5,x_6\},\{x_5,x_7\},\{x_6,x_7\}\}\\
&&
\cup \{\{x_i,x_8\} ~|~ i = 1,\ldots, 7\}
\end{eqnarray*} 
is generated in degrees
$\leq 4$, including
a generator of degree
four, and $\K[G]$ is not Cohen-Macaulay.  (We thank
Kazunori Matsuda for pointing us towards this example.)
\end{remark}

Now we make an observation about the graphs having 
$\deg h_{\K[G]}(x) =1$.

\begin{remark}Let $G=(V,E)$ be a connected and non-bipartite graph such that $h_{\K[G]}(x)=1+ax,$ $a\neq 0.$ From equations \eqref{eq. HS and Betti} and \eqref{eq. reduced HS} in Section \ref{s. background} we have
\[
HS_{\K[G]}(x) = \dfrac{1+ax}{(1-x)^{|V|}}= \frac{1+\sum_j B_j x^{j}}{(1-x)^{|E|} }\ \ \text{where}\ \ B_j=\sum_{i}(-1)^i \beta_{i,j}(\K[G]).
\]
Note that in particular $B_1=0$ and $B_2=-\beta_{1,2}(\K[G])$. 
Thus, we get
$(1+ax)(1-x)^{|E|-|V|}= 1+\sum_j B_j x^{j}.$
So, by comparing coefficients,
$a=|E|- |V|$ and $\beta_{1,2}(I_G)=a^2-\binom{a}{2}=\binom{a+1}{2}=\binom{|E|- |V|+1}{2}$.  So, if there is
a non-bipartite graph $G$ with $\deg h_{\K[G]}(x) =1$,
then it must have $\binom{|E|- |V|+1}{2}$ quadratic
generators.
\end{remark}

\begin{remark}
We know of no example of a graph $G$ such that $\beta_{1,2}(\K[G])=\binom{|E|- |V|+1}{2}$ and $\K[G]$  is not a Cohen-Macaulay ring. 
\end{remark}

Finally, note that the strategy of Theorem \ref{maintheorem} is to find graphs 
where we can control the regularity and the degree
of the $h$-polynomial, and use it as a ``seed" to
repeatedly apply Corollary \ref{p. inductive step reg and degree}.  Thus, to extend Theorem \ref{maintheorem}
for integers $d < r$, we need an appropriate initial
graph.  As the next example shows, we can extend
Theorem \ref{maintheorem} slightly to include
all integers $(r,d)$ with $5 \leq r = d+1$.

\begin{example}Let $Z$ be the graph in Figure \ref{fig:graph Z}  on the vertex set $V=\{x_1, \ldots, x_{10}\}$ and edges $E=$ $\{\{x_1, x_2\},$ $\{x_1, x_3\},$ $  \{x_1, x_7\},$ $ \{x_1, x_8\},$ $ \{x_1, x_9\},$ $\{x_4, x_5\},$ $ \{x_4,  x_6\},$ $ \{x_4, x_7\},$ $ \{x_4, x_8\},$ $ \{x_4, x_9\},$ $\{x_2, x_3\}, $ $ \{x_5, x_6\}, $ $  \{x_7, x_8\}, $ $ \{x_8, x_{10}\},$ $ \{x_9, x_{10}\}\}.$
\begin{figure}[ht]
    \centering
\begin{tikzpicture}
 [scale=.45,auto=left,every node/.style={circle,fill=lightgray, scale=.5, minimum size=3.2em}]
\node (x1) at (3,4) {$x_1$};
\node (x2) at (9,4) {$x_4$};
\node (y1) at (1,1) {$x_7$};
\node (y2) at (6,1) {$x_8$};
\node (y3) at (12,1) {$x_9$};
\node (y6) at (9,0) {$x_{10}$};
\node (z1) at (1,7) {$x_2$};
\node (z2) at (5,7) {$x_3$};
\node (w1) at (7,7) {$x_5$};
\node (w2) at (11,7) {$x_6$};
\draw (z1) -- (z2);
\draw (z2) -- (x1);
\draw (w1) -- (w2);
\draw (w2) -- (x2);
\draw (x1) -- (z1);
\draw (x1) -- (y1);
\draw (x1) -- (y2);
\draw (x1) -- (y3);
\draw (y1) -- (y2);

\draw (y2) -- (y6);
\draw (y6) -- (y3);

\draw (x2) -- (w1);
\draw (x2) -- (y1);
\draw (x2) -- (y2);
\draw (x2) -- (y3);
\end{tikzpicture}
    \caption{The graph $Z$.}
    \label{fig:graph Z}
\end{figure}
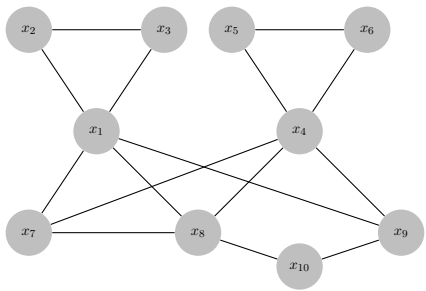
One can compute the Betti diagram and the Hilbert series of $\K[Z]$ by using {\em Macaulay2}~\cite{M2}:
 \[
 \begin{array}{lcr}
\beta(\K[Z])=\begin{matrix}
       &0&1&2&3&4&5&6\\
       \text{total:}&1&12&40&56&37&11&1\\
       \text{0:}&1&\text{.}&\text{.}&\text{.}&\text{.}&\text{.}&\text{.}\\
       \text{1:}&\text{.}&5&5&\text{.}&\text{.}&\text{.}&\text{.}\\
       \text{2:}&\text{.}&2&\text{.}&1&\text{.}&\text{.}&\text{.}\\
       \text{3:}&\text{.}&\text{.}&10&10&\text{.}&\text{.}&\text{.}\\
       \text{4:}&\text{.} &5&25&45&37&10&1\\
       \text{5:}&\text{.}&\text{.}&\text{.}&\text{.}&\text{.}&1&\text{.}\\
       \end{matrix}      
       &\ \ \ \text{and}\  &
       HS_{\K[Z]}(x)= \dfrac{1+5 x+10 x^{2}+13 x^{3}+10 x^{4}}{({1-x})^{10}}.
       \\
  \end{array}
\]
Thus, the graph $Z$ covers the new case 
$\reg(\K[Z])=5$ and $\deg h_{\K[Z]}(x)=4$.
As a consequence of Corollary \ref{p. inductive step reg and degree}, for any pair $(r,d)$ of positive integers such that $d\ge 4$ and $r-d=1$, there is a graph~$G$ with $\reg(\K[G])=r$ and $\deg h_{\K[G]}(x)=d$.\end{example}

Table \ref{summarytable} summarizes all the results from this paper.  In the table, a filled in circle
$\bullet$ denotes a pair $(r,d)$ for which there
is a graph $G$ with $({\rm reg}(\K[G]),\deg h_{\K[G]}(x))
= (r,d)$, the empty circle $\circ$ denotes a pair
$(r,d)$ for which there is no such graph, and the 
unfilled spots denote pairs for which we currently do not
know of a graph that satisfies $({\rm reg}(\K[G]),\deg h_{\K[G]}(x))
= (r,d)$.

\begin{table}[ht]
\begin{tabular}{r|c|c|c|c|c|c|c|ccc}
       &$r=1$& $2$  & $3$  & $4$  & $5$  & $6$  & $7$  &$\cdots$  \\
\hline
$d=1$   & $\bullet$ &     &      &    &    &     &   \\\hline
$2$     & $\circ$   & $\bullet$ &   &    &    &     &     \\\hline
$3$     & $\circ$ &     & $\bullet$  &    &    &     &    \\\hline
$4$     & $\circ$   &    &    &  $\bullet$  &  $\bullet$  &     &    \\\hline
$5$     &  $\circ$  &    &     &  $\bullet$ &  $\bullet$ &  $\bullet$  &    \\\hline
$6$     &  $\circ$ &     &     &  $\bullet$ &  $\bullet$ &  $\bullet$  & $\bullet$  & \\\hline
$\vdots$&  $\vdots$  &    &     &  $\vdots$  &  $\vdots$  &  $\ddots$   & $\ddots$ &$\ddots$  \\
\end{tabular}
%\ \ \ \ \ \ \ \cdot=\emptyset. 
\caption{Summary of comparison of the regularity and the
    degree of the $h$-polynomial.}\label{summarytable}
    \end{table}
\bibliographystyle{plain}

\end{document}